\documentclass[12pt,english,reqno]{amsart}
\usepackage{amscd,amssymb,amsmath,amstext,amsthm,amsfonts,bbold,wasysym}
\usepackage[dvips]{graphicx}
\usepackage{lipsum}
\usepackage{mathrsfs}
\usepackage{appendix}
\usepackage[normalem]{ulem}
\usepackage{float}
\usepackage{hyperref}
\usepackage[dvipsnames]{xcolor}
\usepackage[ansinew]{inputenc}
\usepackage{enumerate}

\usepackage{a4wide}
\newcommand{\cv}{\vec c}

\newcommand\supsetsim{\mathrel{
  \ooalign{\raise0.2ex\hbox{$\supset$}\cr\hidewidth\raise-0.8ex\hbox{\scalebox{0.9}{$\sim$}}\hidewidth\cr}}}

\newtheorem{Theorem}{Theorem}

\newtheorem{maintheorem}{Theorem}

\newtheorem{T}{Theorem}[section]
\newtheorem{Corollary}[T]{Corollary}
\newtheorem{Proposition}[T]{Proposition}
\newtheorem*{prop}{Proposition}
\newtheorem{Lemma}[T]{Lemma}

\newtheorem{Remark}[T]{Remark}

\newtheorem{Claim}{Claim}

\def \BB {{\mathbb B}}

\def \RR {{\mathbb R}}
\def \ZZ {{\mathbb Z}}
\def \NN {{\mathbb N}}

\def \XX {{\mathbb X}}

\def \cm {\mathcal{M}}

\def \cp {\mathcal{P}}

\def \cc {\mathcal{C}}

\def \cu {\mathcal{U}}
\def \co {\mathcal{O}}

\def \cv {\mathcal{V}}

\def \cw {\mathcal{W}}

\newcommand{\per}{\operatorname{Per}}

\newcommand{\interior}{\operatorname{interior}}
\newcommand{\fix}{\operatorname{Fix}}
\newcommand{\diam}{\operatorname{diameter}}

\begin{document}


\author{V. Pinheiro}
\address{Instituto de Matematica - UFBA, Av. Ademar de Barros, s/n,
40170-110 Salvador Bahia}
\email{viltonj@ufba.br}


\date{\today}

\title{Topological and statistical attractors for interval maps}

\maketitle

\begin{abstract}
We use the concept of Baire Ergodicity and Ergodic Formalism, introduced in \cite{Pi21}, to study topological and statistical attractors for interval maps, even with discontinuities.
For that we also analyze the {\em wandering intervals attractors}. As a result, we establish the finiteness of the non-periodic attractors for maps $C^2$ with discontinuities.

For $C^2$ interval maps without discontinuities, we show the coincidence of the statistical attractors  with the topological ones and we calculate  the upper Birkhoff averages of continuous functions for generic points, even when the map has abundance of historical behavior.
\end{abstract}
\tableofcontents

\section{Introduction}

Since Poincaré and Denjoy, the existence of wandering intervals has been a central and delicate problem in One Dimensional  Dynamics.
In particular, a difficulty to the understanding of the asymptotical behavior of the orbit of most points and to prove the finiteness of attractors.
In \cite{Gu} Guckenheimer proves that non-flat S-unimodal maps does not admit wandering interval and that such a maps has a unique topological attractor.

The results of Guckenheimer ware generalized by Blokh and Lyubich for non-flat $S$-multimodal maps \cite{BL} and by Lyubich \cite{Ly89} for non-flat $C^3$ multimodal maps.
In both papers Blokh and Lyubich show the absence of wandering intervals (the mains theorem of both papers) and the finiteness of topological attractors (Corollary~2.2 in \cite{BL} and Corollary~2.2 in \cite{Ly89}).
For $C^2$ non-flat maps (not only multimodal maps), the absence of wandering interval was proved by de Melo and van Strien (see Theorem~A in chapter IV at \cite{MvS}, a previous proof for $C^3$ maps appeared in \cite{MS89}). 
In \cite{vSV} van Strien and Vargas  proved the finiteness of non-periodic metrical attractors for $C^2$ non-flat maps and obtained an alternative proof for the absence of wandering intervals for $C^2$ non-flat maps.
All the papers cited above use a deep understanding of the asymptotical behavior of those dynamics  and the absence of wandering interval to obtain the finiteness of non-periodic topological or metric attractors.

In contrast with the smooth case, it is well known that interval maps with discontinuities may have wandering interval associated to Cherry attractors.
It was conjectured (see for instance Conjecture 1.14. at \cite{MM}) that  all wandering intervals for non-flat piecewise $C^2$ maps belong to the basin of a Cherry attractors, but this conjecture is open even for contracting Lorenz maps (maps with a unique discontinuity).
In \cite{Br}, Brandão was able to show the uniqueness of non-periodic topological attractors for $C^2$ non-flat contracting Lorenz maps. 
The same result for metrical attractors was obtained by Keller and St.~Pierre \cite{SP,KsP} assuming the Lorenz map being $C^3$ non-flat and with negative Schwarzian derivative.
In \cite{BPP21}, Brandão, Palis and Pinheiro proved Keller and St.~Pierre's result without assuming  non-flatness.

In the papers above about contracting Lorenz maps, the existence of only one discontinuity and the preserving orientation property of the maps are crucial in the proofs.
Nevertheless, in \cite{BPP19} Brandão, Palis and Pinheiro show that the existence of wandering interval is not an obstruction to prove the finiteness of non-periodic metric attractors for general maps even with several discontinuities.
This fact is also true for topological attractors (Proposition~\ref{PropositionWanderingIntervals}) and so, combining this information about wandering intervals with the Ergodic Formalism,
we were able to prove the finiteness of non-periodic attractor for discontinuous maps that are piecewise $C^2$ (see Theorem~\ref{TheoremFinTopInterC2} in the next section).
For $C^2$ interval maps without discontinuities, we used Baire ergodicity to analyze the statistical behavior of generic points of the interval, showing the coincidence of the statistical attractors (defined by Ilyashenko \cite{Ily}) with the topological ones and calculating the upper Birkhoff averages of continuous functions for generic points, even when the map has abundance of historical behavior (see Theorem~\ref{TheoremUdecompodition} in the next section). 

\section{Basic notations and statement of mains results}

A countable union of nowhere dense subsets of $[0,1]$ is said to be {\bf \em meager}; the complement of such a set contains a {\bf \em residual set}, i.e., a countable intersection of open and dense sets. If a set is not meager then it is called {\bf \em fat}.
We say that two sets $U,V\subset\XX$ are (meager) {\bf\em  equivalent}, denoted by $U\sim V$, when $U\triangle V:=(U\setminus V)\cup(V\setminus U)$ is a meager set. We say that a given property $P$ is {\bf\em generic on $U\subset\XX$} if the set of all $x\in U$ where $P$ is not true is a meager set.

Let $\cc\subset(0,1)$ be a finite set and $f:[0,1]\setminus\cc\to[0,1]$ a $C^2$ local diffeomorphism. 
We say that $f$ is {\bf\em non-flat} if for each $c\in\cc$  there exist $\varepsilon>0$, constants $\alpha, \beta\ge1$  and $C^2$  diffeomorphisms $\phi_{0}:[c-\varepsilon,c]\to \text{Im}(\phi_{0})$ and $\phi_{1}:[c,c+\varepsilon]\to\text{Im}(\phi_{1})$ such that $\phi_1(c)=\phi_2(c)=0$ and $$f(x)=\begin{cases}a+\big(\phi_{0}(x)\big)^\alpha&\text{ if }x\in(c-\varepsilon,c)\cap(0,1)\\
b+\big(\phi_{1}(x)\big)^\beta&\text{ if }x\in(c,c+\varepsilon)\cap(0,1)\end{cases}$$
where $a=f(c_-):=\lim_{0<\delta\to 0}f(c-\delta)$ and $b=f(c_+):=\lim_{0<\delta\to 0}f(c+\delta)$. We say that $f$ is {\bf\em non-degenerated} if it is non-flat and $\#\fix(f^n)<+\infty$ for every $n\in\NN$
(\footnote{Notice that every non-flat piecewise $C^3$ interval map $f$ with negative Schwarzian derivative is non-degenerated.
Furthermore, the property of being non-degenerated is a generic property for a $C^2$ interval map.}).

The {\bf\em pre-orbit} of $A\subset[0,1]$ is $\co_f^-(x):=\bigcup_{n\ge0}f^{-n}(A)$ and the pre-orbit of $x\in[0,i]$ is $\co_f^-(x):=\co_f^{-}(\{x\})$.
The {\bf\em forward orbit} of $x\notin\co_f^-(\cc)$ is $\co_f^+(x)=\{f^n(x)\,;\,n\ge0\}$. If $x\in\co_f^-(\cc)$, set $\co_f^+(x)=\{x,\cdots,f^{n}(x)\}$, where $n=\min\{j\ge0\,;\,x\in f^{-j}(\cc)\}$. 
If $x\notin\co_f^-(\cc)$, then the {\bf\em omega-limit} of $x$, the accumulating points of the forward orbit of $x$, is  $\omega_f(x)=\bigcap_{n\ge0}\overline{\co_f^+(f^n(x))}\ne\emptyset$.
If $x\in\co_f^-(\cc)$,  $\omega_f(x)=\emptyset$.
The {\bf\em alpha limit set} of $x$, the accumulating points of the pre-orbit of $x$, is
$\alpha_f(x)=\bigcap_{n\ge0}\overline{\co_f^-(f^{-n}(x))}$.

A {\bf\em forward invariant set} is a set $A\subset[0,1]$ such that $f(A\setminus\cc)\subset A$. Given a compact forward invariant set $A$, define the {\bf\em basin of attraction of $A$} as $$\beta_f(A)=\{x\in\XX\,;\,\emptyset\ne\omega_f(x)\subset A\}.$$
Following Milnor's definition of topological attractor (indeed, minimal topological 
attractor \cite{Mi}), a compact forward invariant set $A$
is called  {\bf\em a topological attractor}, or for short, an attractor, if
$\beta_f({A})$ and $\beta_f({A})\setminus\beta_f(A')$ are fat sets for every compact forward invariant set $A'\subsetneqq A$.

\begin{maintheorem}
[Finiteness of non-periodic topological attractors for piecewise $C^2$ maps]\label{TheoremFinTopInterC2}
If $\cc\subset(0,1)$ is a finite set and $f:[0,1]\setminus\cc\to[0,1]$ is a  non-degenerated $C^2$ local diffeomorphism, then  
either $\BB(f)$, the union of the basin of attraction of all periodic-like attractor, is an open and dense subset of $[0,1]$ or there exists a finite collection of attractors $A_1,\cdots,A_n$, with $n\le\#\cc+2^{2\#\cc}$, such that 
\begin{center}
$[0,1]\setminus\BB(f)\sim\bigcup_{j=1}^{\ell}\beta_f(A_j)$.
\end{center} 
Furthermore,
\begin{enumerate}
\item for every $1\le j\le n$, $\omega_f(x)=A_j$ generically on $\beta_f(A_j)$;
\item if $f$ does not admit wandering interval then $n\le\#\cc$;
\item each $A_j$ is either a Cantor set or a cycle of intervals;
\item if $A_j$ is a Cantor set then  $A_j=\bigcup_{v\in V}\overline{\co_f^+(v)}$, for some $V\subset\{f(c_\pm)\,;\,c\in\cc\}$ $($\footnote{ Indeed, $V=\{f(c_{\pm})\,;\,c\in\cc_{\pm}\}$, where $\cc_-=\{c\in\cc\,;\,c\in\overline{\co_f^+(x)\cap(0,c)}\}$ generically on $\beta_f(A_j)$ and $\cc_+=\{c\in\cc\,;\,c\in\overline{\co_f^+(x)\cap(c,1)}\}$ generically on $\beta_f(A_j)$.}$)$;
\item if $I$ is a wandering interval then $I\subset\beta_f(A_j)$ for some $1\le j\le\ell$, where $A_j$ is a Cantor set;
\end{enumerate}
\end{maintheorem}

The (upper) {\bf\em   visiting frequency} of $x\in\XX$ to $V\subset\XX$ is given by
\begin{equation}\label{Equationhgfdfb345}
  \tau_x(V)=\tau_{x,f}(V)=\limsup_{n\to\infty}\frac{1}{n}\#\{0\le j<n\,;\,f^{j}(x)\in V\}.
\end{equation}
We define the {\bf\em statistical $\omega$-limit set} of $x\in\XX$ as
$$\omega_f^*(x)=\{y\,;\,\tau_x(B_{\varepsilon}(y))>0\text{ for all }\varepsilon>0\}.$$
According to Ilyashenko (see page 148 of \cite{AAIS}), the {\bf\em statistical basin of attraction} of a compact forward invariant set $A\subset\XX$ is defined as  $$\beta_f^*(A)=\{x\,;\,\omega_f^*(x)\subset A\}.$$ We call $A$ a (topological) {\bf\em statistical attractor} for $f$ if $\beta_f^*(f)$ is a fat set and there is no compact set $A' \subsetneqq A$ such that $\beta_f^*(A')$ is the same as $\beta_f^*(A)$ up to a meager set (\footnote{ If, instead of $\beta_f^*(A)$ being fat, we ask that $\beta_f^*(A)$ has positive Lebesgue measure then $A$ is called a {\bf\em (metrical) statistical attractor}.}).

According to Ruelle \cite{Ru} and Takens \cite{Ta08}, a point $x\in\XX$ has {\bf\em historic behavior} when $\limsup_n\frac{1}{n}\sum_{j=0}^{n-1}\varphi\circ f^j(x)>\liminf_n\frac{1}{n}\sum_{j=0}^{n-1}\varphi\circ f^j(x)$ for some continuous function $\varphi:\XX\to\RR$.

\begin{maintheorem}[Generic statistical behavior  for $C^2$ interval maps]\label{TheoremUdecompodition}
If $f:[0,1]\to[0,1]$ is a non-degenerated $C^2$ map with critical set $\cc$ then there exist a finite collection of statistical attractors $A_1,\cdots,A_m$ such that:
\begin{enumerate}
\item $\bigcup_{j=1}^m\beta_f^*(A_j)$ contains a residual subset of $[0,1]$;
\item the number of attractors that aren't  periodic-like is smaller or equal to $\#\cc$;
\item $\omega_f(x)=\omega_f^*(x)=A_j$ generically on $\beta_f^*(A_j)$, that is, the topological and the statistical attractors are the same, in particular
\begin{enumerate}
\item[(3.1)] $\beta_f^*(A_j)=\beta_f(A_j)$;
\item[(3.2)] each $A_j$ is either a periodic orbit or a Cantor set or a cycle of intervals;
\item[(3.3)] $\omega_f(x)=\omega_f^*(x)\;\text{ generically on }[0,1]$;
\end{enumerate}

\item a generic point in $\beta_f^*(A_j)$ has historic behavior $\iff$ $A_j$ is a cycle of intervals;
\item $h_{top}(f|_{A_j})>0$ $\iff$ $A_j$ is a cycle of intervals;
\item generically on $[0,1]$ the (upper) Birkhoff average of a continuous function $\varphi\in C([0,1],\RR)$ admits only a finite numbers of values. That is, setting
\begin{center}
$a_j=\max\{\int\varphi d\mu\,;\,\mu\in\cm^1(f)\text{ and }\mu(A_j)=1\}$,	
\end{center}
then
$$\limsup\frac{1}{n}\sum_{j=0}^{n-1}\varphi\circ f^j(x)\in\{a_1,\cdots,a_s\}\;\text{ generically on }[0,1].$$
Indeed, $\limsup\frac{1}{n}\sum_{j=0}^{n-1}\varphi\circ f^j(x)=a_j$ generically on $\beta_f^*(A_j)$.
\end{enumerate}
\end{maintheorem}

\section{Baire ergodicity}

In this section, let $\cc\subset(0,1)$ be a finite set and $f:[0,1]\setminus\cc\to[0,1]$ a local homeomorphism. 
A set $A\subset[0,1]$ is called a  {\bf  \em   Baire set} if there is an open set $U$ such that $A\triangle U:=(A\setminus U)\cup(U\setminus A)$ is a meager set, that is, if $A$ is equal to an open set modulo a meager sets.
A Baire set is also called {\bf  \em an almost open set} or a set with the {\bf \em Baire property} ({\bf BP}).
Hence, a Baire set $A\subset X$ is fat if and only there exists a nonempty open set $U\subset X$ such that $A=U$ modulo a meager set.
The collection of all Baire sets is a $\sigma$-algebra. Indeed, it is the smallest $\sigma$-algebra containing all open sets and all meager sets. Moreover, in $[0,1]$, the Borel sets and  Baire sets coincide.

As $f$ is a local homeomorphism, $f^{-1}(U)$ is meager whenever $U$ is meager, i.e., $f$ is a {\bf\em non-singular map}.
Let $2^{[0,1]}$ be the power set of $\XX$, that is, the set for all subsets of $\XX$, including the empty set. Define $\bar{f}:2^{[0,1]}\circlearrowleft $ given by
$$\bar{f}(U)=\begin{cases}
	\emptyset & \text{ if }U\subset\cc\\
	f(U\setminus\cc)=\{f(x)\,;\,x\in U\setminus\cc\} & \text{ if }U\not\subset\cc
\end{cases}$$
We say that $U\subset\XX$ is {\bf\em invariant} if $f^{-1}(U)=U$,  {\bf\em forward invariant} if $\bar{f}(U)\subset U$, {\bf\em almost invariant} if $f^{-1}(U)\sim U$ and {\bf\em almost froward invariant} when $\bar{f}(U)\sim U$.

It was introduced in \cite{Pi21} the concept of {\em Baire Ergodicity} for general Baire topological spaces.
In our context, if $U\subset[0,1]$ is a measurable fat set then $g:U\to[0,1]$ is called {\bf\em Baire ergodic} if $g^{-1}(U)\sim U$ and $V\sim\emptyset$ or $V\sim U$ for every invariant measurable set $V\subset U$.
We say that a measurable fat set $U\subset\XX$ is a {\bf \em Baire ergodic component for $f$} if $U\sim f^{-1}(U)$ and  $f|_{U}$ is a Baire ergodic map.
The main connection between Baire ergodic components and attractors is the following result.

\begin{Proposition}[See Proposition~4.2 and 4.7 at \cite{Pi21}]\label{Propositiontop-ergodicAttractors}
If $U\subset \XX$ is a Baire ergodic component of $f$, then there exists a unique topological attractor $A$ and a unique statistical attractor $A^*$ for $U$.
Moreover,  $A^*=\omega_f^*(x)\subset\omega_f(x)=A$ generically on $U$ and $U\sim\beta_f(A)\supset\beta_f^*(A)\sim U$.
\end{Proposition}

A fat set $U\subset[0,1]$ is called {\bf\em asymptotically transitive} if 
$$
\bigg(\bigcup_{j\ge0}\bar{f}^j(A)\bigg) \cap \bigg(\bigcup_{j\ge0}\bar{f}^j(B)\bigg)\not\sim\emptyset$$ for every open set (with respect to the induced topology) $A$ and $B\subset U$.
A it was proved in \cite{Pi21}, in most case, asymptotically transitive is equivalent to be Baire ergodic.

\begin{Lemma}[Corollary of Proposition~3.8 of \cite{Pi21}]\label{Lemmaoiuytr567}
Let $f:[0,1]\setminus\cc\to[0,1]$ be a local homeomorphism, where $\cc$ is a finite subset of $(0,1)$ and let $U\subset[0,1]$ be an almost invariant measurable set.
If $U$ is a fat set then $U$ is a Baire ergodic component of $f$  if and only if $U$ is asymptotically transitive.
\end{Lemma}
\begin{proof}
As $\cc$ is finite and $f$ is a local homeomorphism, we get that $$U'=\bigcup_{n\ge0}f^{-n}\bigg(\bigcap_{j\ge0}f^{-j}(U)\bigg)$$ is a $f$-invariant set and $U'\sim U$ (see Lemma~4.1 in \cite{Pi21}).
As $U'$ is a Borel set, it is a Baire set and a Baire space with respect to the induced topology.
Moreover, $f|_{U'}$ is continuous and non-singular. Hence, it follows from Proposition~3.8 of \cite{Pi21} that $f|_{U'}$ is Baire ergodic if and only if $f|_{U'}$ is asymptotically transitive.
As $U'\sim U$, ($U$ is a Baire ergodic component of $f$) $\iff$ ($f|_{U'}$ is Baire ergodic) $\iff$  ($f|_{U'}$ is asymptotically transitive) $\iff$ ($U$ is asymptotically transitive).

\end{proof}

\section{Wandering intervals and attractors}

Let $\cc$ be a finite subset of $(0,1)$ and $f:[0,1]\setminus\cc\to[0,1]$ be a local homeomorphism. We say that $p\in[0,1]$ is a periodic-like point if there is a $n\ge1$ such that
\begin{enumerate}
\item $f^n((p-\delta,p))\cap(p-\delta,p)\ne\emptyset$ for every $\delta>0$ or
\item $f^n((p,p+\delta))\cap(p,p+\delta)\ne\emptyset$ for every $\delta>0$.
\end{enumerate}
We say that $f^n(p_-)=p_-$ when the first case above is true. If second case is true, we say that $f^n(p_+)=p_+$. 

A {\bf\em periodic-like attractor} is a finite set $A=\{p,f(p_-),\cdots,f^{n-1}(p_-)\}$, with $f^n(p_-)=p_-$, or $A=\{p,f(p_+),\cdots,f^{n-1}(p_+)\}$, with $f^n(p_+)=p_+$. such that 
\begin{enumerate}
\item $f^n|_{(p-r,p)}$ is a homeomorphism, $f^n((p-r,p))\subset(p-r,p)$ and $\lim_jf^{n j}(x)=p$ for every $x\in(p-r,p)$ or
\item $f^n|_{(p,p+r)}$ is a homeomorphism, $f^n((p,p+r))\subset(p,p+r)$ and $\lim_jf^{n j}(x)=p$ for every $x\in(p,p+r)$,
\end{enumerate}
for some $r>0$. Denote the {\bf\em union of the basin of attraction of all periodic-like attractors} by $\BB(f)$ and let $\per(f)$ be the {\bf\em set of all periodic-like points of $f$}.

\begin{Remark}\label{RemarkPerLike}
A periodic-like point $p$ is not a periodic point only if $\co_f^+(p_\pm)\cap\cc\ne\emptyset$. Hence there exist at most $2\#\cc$ periodic-like orbits that are not periodic orbits. In particular, there are at most $2\#\cc$ periodic-like attractors.
\end{Remark}

A {\bf\em homterval} is an open interval $J=(a,b)$ such that $f^n|_J$ is a homeomorphism for every $n\ge1$. A {\bf\em wandering interval} is a homterval $J$ such that $J\cap\BB(f)=\emptyset$ and $f^j(J)\cap f^k(J)=\emptyset$ for all $1\le j<k$. Recall that the {\bf\em nonwandering  set} of $f$, denoted by $\Omega(f)$, is the set of point $x\in X$ such that $V\cap\bigcup_{n\ge0}\bar{f}^n(V)\ne\emptyset$ for every open neighborhood $V$ of $x$. Hence $I\cap\Omega(f)=\emptyset$ for every wandering interval $I$.

\begin{Lemma}[Homterval Lemma, see for instance Lemma~3.1 at \cite{MvS}]\label{LemmaHomterval}
Let $U$ be an open subset of $[0,1]$, $f:U\to[0,1]$ a continuous map and $I=(a,b)$ a homterval of $f$. If $I$ is not a wandering interval then $I\subset\BB(f)\cup\co_f^-(\per(f))$.
\end{Lemma}

It follows from the Homterval Lemma that, if $I=(a,b)$ is not a homterval then, taking $n=\min\{j\ge0\,;\,f^j(I)\cap\cc\ne\emptyset\}$ then  $c\in\cc$ such that $c\in f^n(I)$ and $f^n|_I$ is a homeomorphism between $I$ and the open interval $f^n(I)$. As a consequence, if $I=(a,b)$ is not a homterval then $$\interior\bigg(\bigcup_{n\ge0}f^n(I))\bigg)\cap\cc\ne\emptyset.$$

\begin{Lemma}\label{LemmaBaireXhomterval}
	If $U\subset[0,1]$ is a Baire ergodic component of a local homeomorphism $f:[0,1]\setminus\cc\to[0,1]$, where $\cc$ is a finite subset of $(0,1)$, and $I$ is a homterval for $f$ then $U\cap I\sim\emptyset$.
\end{Lemma}
\begin{proof}
If $I$ is a wandering interval and $U\cap I$ is a fat set then let $(p,q)\subset I$ be such that $U\cap(p,q)\sim(p,q)$. As $(p,q)$ is a wandering interval, it follows from the Homterval Lemma that $(p,q)\subset\BB(f)\cup\co_f^-(\per(f))$.
Thus, $U_0=\bigcup_{n\in\ZZ}f^{n}((p,(p+q)/2))$ and $U_1=\bigcup_{n\in\ZZ}f^{n}((p+q)/2,q))$ are invariant fat sets and $U_0\cap U_1=\emptyset$, but this is impossible since we are assuming that $U$ is a Baire ergodic component.
\end{proof}

Proposition~\ref{PropositionOMEGAdecomposition} can be seen as a topological version of Theorem~1 and 2 of \cite{Ly91}.

\begin{Proposition}\label{PropositionOMEGAdecomposition}
Let $\cc\subset(0,1)$ be a finite set, $f:[0,1]\setminus\cc\to[0,1]$ be a local homeomorphism and $H(f)$ the union of all homtervals of $f$.
If $[0,1]\setminus H(f)$ is a fat set then $[0,1]\setminus H(f)$ can be decomposed into at most $\#\cc$ Baire ergodic components. 

Moreover, if $U$ is a Baire ergodic component of $f$ then  $U\cap H(f)\sim\emptyset$ and $U$ is residual in $B_{\varepsilon}(c)$ for some $c\in\cc$ and $\varepsilon>0$.
\end{Proposition}
\begin{proof} Let $M\subset[0,1]$ be an open set such that $M\sim[0,1]\setminus H(f)$ and set $\XX=\overline{M}$.
Note that if $x\in\XX$ then $\bigcup_{n\ge0}\bar{f}^n(B_{\varepsilon}(x))\cap\cc\ne\emptyset$ for every $\varepsilon>0$. 
Thus, 
for each $c\in\cc$, define
\begin{equation}\label{Equationiutyc9}
  \cu(c)=\bigg\{x\in\Omega(f)\,;\,B_{\varepsilon}(x)\cap\Omega(f)\not\sim\emptyset\text{ and }c\in \interior\bigg(\bigcup_{n\ge0}\bar{f}^n(B_{\varepsilon}(x))\bigg)\;\forall\,\varepsilon>0\bigg\}.
\end{equation}

Hence,  $\XX=\bigcup_{c\in\cc}\cu(c)$ and therefore at least one element of $\{\cu(c);c\in\cc\}$ is a fat set.
Moreover, if $\cu(c)$ is a fat set then $c\in\Omega(f)$.
Let $g:\XX\setminus\cc\to\XX$ be the restriction of $f$ to $\XX$ and consider on $\XX$ the induced topology.

\begin{Claim}
If $\cu(c)$ is a fat set then it is a Baire ergodic component for $g$.
\end{Claim}
\begin{proof}[Proof of the claim]
Note that $g^{-1}(\cu(c))=\cu(c)$ and, as $\cu(c)$ is compact, $\cu(c)$ is  measurable.
Of course that, $c\in\interior\big(\bigcup_{n\ge0}\bar{g}^n(V_0)\big)\cap\interior\big(\bigcup_{n\ge0}\bar{g}^n(V_1)\big)$ for every open sets $V_0,V_1\subset\cu(c)$, proving that $\cu(c)$ is asymptotically transitive and so, by Lemma~\ref{Lemmaoiuytr567}, $\cu(c)$ is a Baire ergodic component to $g$.
\end{proof}

Note that when $\cu(c)$ is a fat compact set, it follows from the definition of $\cu(c)$ and from $g^{-1}(\cu(c))=\cu(c)$,  that $B_{\varepsilon}(c)\subset\cu(c)$ for some $\varepsilon>0$.

Let $\cc\subset\cc_0:=\{c\in\cc\,;\,\cu(c)$ is a fat set$\}\subset\cc$ 
such that $\{\cu(c)\,;\,c\in\cc\}=\{\cu(c)\,;\,c\in\cc_0\}$.
Let $c,c'\in\cc_0$. As $\cu(c)$ and $\cu(c')$ are Baire ergodic components, either $\cu(c)\sim\cu(c')$ or $\cu(c)\cap\cu(c')\sim\emptyset$.
Thus, we can choose $\{c_1,\cdots,c_\ell\}\subset\cc_0$ such that for each $c\in\cc_0$ there exists a unique $c_j$ such that $\cu(c)\sim\cu(c_j)$.
Therefore, $\{\cu(c_1),\cdots,\cu(c_\ell)\}$ is a Baire ergodic decomposition of $[0,1]\setminus H(f)$ for $f$.
Noting that $\ell\le\#\cc$ and that $H(f)\cap U\sim\emptyset$ for any Baire ergodic component of $f$ (Lemma~\ref{LemmaBaireXhomterval}), we conclude the proof.
\end{proof}

\begin{Corollary}[$\Omega$ decomposition for interval maps]\label{CorollaryOMEGAdecomposition}
Let $\cc\subset(0,1)$ be a finite set and $f:[0,1]\setminus\cc\to[0,1]$ be a local homeomorphism such that $\per(f)$ is a countable set.
If $\Omega(f)$ is a fat set then $\Omega(f)$ can be decomposed into at most $\#\cc$ Baire ergodic components. 
\end{Corollary}
\begin{proof}
Let, as in Proposition~\ref{PropositionOMEGAdecomposition},  $H(f)$ be the union of all homtervals of $f$.
	As $\Omega(f)\cap H(f)=\per(f)$. If $\per(f)\sim\emptyset$, we get that $\Omega(f)\cap H(f)\sim\emptyset$ and so, the Baire ergodic decomposition of $[0,1]\setminus H(f)$ given by Proposition~\ref{PropositionOMEGAdecomposition} induces a Baire ergodic decomposition of $\Omega(f)$.
\end{proof}

%
%
%
%
%
%

\begin{Lemma}\label{LemmaINTERVALDICHOTOMY}
Let $U\subset(a,b)\subset[0,1]$ be a nonempty open set and $\cp$ be the collection of all connected components of $U$.
Let $F:U\to[a,b]$ be a {\em full branch map}, i.e., $F$ is a local homeomorphism and $F(P)=(a,b)$ for every $P\in\cp$. 
If $\overline{U_0}$ has nonempty interior, where $U_0:=\bigcap_{n\ge0}F^{-n}((a,b))$, then either $\omega_F(x)=[a,b]$ generically on $(a,b)$ or $\overline{U_0}\sim H$, where $H$ is a union of homtervals. 
\end{Lemma}
\begin{proof}
Let $\cp$ be the collection of all connected components of $U$. Given $n\in\NN$ and  $x\in F^{-n}((a,b))$, let $\cp_n(x)$ be the connected component of $F^{-n}((a,b))$ containing $x$.
Note that 
	if $\lim_n\diam(\cp_n(x))>0$ for some $x\in U_0$, then $I:=\bigcap_{n\in\NN}\cp_n(x)$ is a homterval.
Indeed, as $I\subset\cp_n(x)$ and $F^n|_{\cp_n(x)}$ is a homeomorphism between $\cp_n(x)$ and $(a,b)$ for all $n\in\NN$, we get that $F^n|_I$ is a homeomorphism between $I$ and $F^n(I)$ for every $n\in\NN$.
So, as $I$ is a interval, it is a homterval. In particular, for each $n\in\NN$, there is $P_n\in\cp$ such that $F^n(I)\subset P_n$. 

Therefore, $V:=\{x\in U_0\,;\,\lim_n\diam(\cp_n(x))>0\}$ is a subset of $H$.
Let $T=U_0\setminus V$.
If $\overline{T}\sim\emptyset$ the proof is finished.
Hence, we may assume that $B:=\interior(\overline{T})\ne\emptyset$.

Choose $p\in B\cap T$.
Taking $n\in\NN$ big enough, we get $\cp_n(x)\subset B$.
As $F^n|_{\cp_n(x)}$ is a homeomorphism of $\cp_n(x)$ onto $(a,b)$, it follows from the forward invariance of $T$ that $T$ is dense in $(a,b)$.
Moreover, it follows from that fact that $\lim_n\diam(\cp_n(y))=0$ for every $y\in B\cap T$ that $V=\emptyset$.
Indeed, if $p\in V$, let $J:=\bigcap_{n\ge1}\cp_n(p)$ and note that $\bigcap_{n\ge1}\cp_n(x)=J$ for every $x\in J$, proving that $J\subset V$, but this is impossible as $T$ is dense in $(a,b)$.

Finally, as $U_0=T$ is dense in $(a,b)$, $\diam(\cp_n(x))\to0$ $\forall\,x\in U_0$  and $F^n(\cp_n(x))=(a,b)$ always, we get  that $F$ is transitive and, so, $\omega_F(x)=[a,b]$ residually.
\end{proof}

\begin{Proposition}\label{PropositionAttractorDichotomy}
If $U\subset[0,1]$ is a Baire ergodic component of a local homeomorphism $f:[0,1]\setminus\cc\to[0,1]$, where $\cc$ is a finite subset of $(0,1)$, then there is a  topological attractor $A$ of such that:
\begin{enumerate}
\item $\beta_f(A)\sim U$;
\item $\omega_f(x)=A$ generically on $U$;
\item $A$ is either a Cycle of intervals of a Cantor set;
\item if $A$ is a Cantor set then $$A=\bigg(\bigcup_{c\in\cc_-}\overline{\co_f^+(f(c_-))}\bigg)\cup\bigg(\bigcup_{c\in\cc_+}\overline{\co_f^+((c_+))}\bigg),$$ where  $\cc_-=\{c\in\cc\,;\,c\in\overline{\co_f^+(x)\cap(0,c)}\}$ generically on $U$ and $\cc_+=\{c\in\cc\,;\,c\in\overline{\co_f^+(x)\cap(c,1)}\}$ generically on $U$.

\end{enumerate}
\end{Proposition}
\begin{proof}
The existence of the topological attractor follows from Proposition~\ref{Propositiontop-ergodicAttractors}.
From the same proposition we get that $\beta_f(A)\sim U$ and that $\omega_f(x)=A$ generically on $U$, proving items (1) and (2).

As $\omega_f(x)=A$ generically on $U$, if $A$ has nonempty interior, we get $\co_f^+(x)\cap\interior(A)\ne\emptyset$ generically on $U$.
This implies that $\co_f^+(x)$ is dense on $A$ generically on $\interior(A)$.
Hence, $f|_A$ is transitive and $A=\overline{\interior(A)}$ and so, $A$ is a countable union of non trivial closed interval and, by compacity, $A$ is a finite union of close d intervals, proving that $A$ is a cycle of intervals.

Suppose that $\interior(A)=\emptyset$.
It follows from the proof of Proposition~\ref{PropositionOMEGAdecomposition} that $U\subset\alpha_f(c)$ for some $c\in\cc$.
This implies that $c\in\omega_f(x)$ generically on $U$ and so, $\omega_f(x)\supset\omega_f(f(c_-))$ or $\omega_f(x)\supset\omega_f(f(c_+))$.

%
%

For each $c\in\cc\setminus\cc_-$ let $0<a_c<c$ be such that $(a_c,c)\cap\cc=\emptyset$ and $(a_c,c)\cap\co_f^+(x)=\emptyset$ generically on $U$.
Likewise, for each $c\in\cc\setminus\cc_+$ let $c<b_c<1$ be such that $(c,b_c)\cap\cc=\emptyset$ and $(c,b_c)\cap\co_f^+(x)=\emptyset$ generically on $U$.
Let $g:[0,1]\setminus\cc\to[0,1]$ be a local homeomorphism such that $g(c_-)\in\{0,1\}$ for every $c\in\cc\setminus\cc_-$, $g(c_+)\in\{0,1\}$ for every $c\in\cc\setminus\cc_+$ and $g(x)=f(x)$ for every $x\notin V:=\big(\bigcup_{c\in\cc\setminus\cc_-}(a_c,c)\big)\cup\big(\bigcup_{c\in\cc\setminus\cc_+}(c,b_c)\big)$.

As $\co_f^+(x)=\co_g^+(x)$ generically on $U$, we get that $U$ is a Baire ergodic component to $g$, $A$ is the topological attractor associated to $U$ with respect to $g$ and $\omega_g(x)=A=\omega_f(x)$ generically on $U$.

Hence, defining $A_0:=\big(\bigcup_{c\in\cc_-}\overline{\co_f^+(f(c_-))}\big)\cup\big(\bigcup_{c\in\cc_+}\overline{\co_f^+(f(c_+))}\big)$, we have that $A_0\subset\omega_f(x)=A$ generically on $U$.
Suppose $p\in A\setminus A_0$ and that $\cw(p)=\{x\in\ U\,;\,p\in\omega_g(x)\}$ is a fat set.
Let $I=(a,b)$ be the connected component of $[0,1]\setminus A_0$ containing $p$.
Note that $\co_g^+(\partial I)\cap I=\emptyset$.
That is, $I$ is a ``nice interval''.
Let $G:I^*\to I$ be the first return map to $I$ by $g$, where $I^*=\{x\in I\,;\,\co_g^+(g(x))\cap I\ne\emptyset\}$.
Note that $\co_g^+(g(c_\pm))\cap I=\emptyset$.
Indeed, if $c\in\cc_\pm$ then $\overline{\co_g^+(g(c_\pm))}=\overline{\co_f^+(f(c_\pm))}\subset A_0$ and if $c\in\cc\setminus\cc_\pm$ then $\overline{\co_g^+(g(c_\pm))}\subset\{0,1\}$.
Therefore, $G$ is a full branch map.
It follows from $g$ being a non-singular map and from $p\in\omega_g(x)\cap I$ that 
$U_0:=U\cap I\sim\bigcap_{n\ge0}G^{-n}(U)$ $=$ $\{x\in U\cap I\,;\,\#(\co_g^+(x)\cap I)=\infty\}$ $\subset$ $\bigcap_{n\ge0}G^{-n}(I)$ is a  fat set. 
Hence, it follows from Lemma~\ref{LemmaINTERVALDICHOTOMY} that either (i) $\omega_G(x)=[a,b]$ generically on $(a,b)$ or (ii) $U\cap I\sim H$, where $H$ is a union of homtervals for $G$. As $\omega_G(x)=[a,b]$ implies that $\omega_g(x)$ is a cycle of interval, (i) cannot occur.
However, as every homterval for $G$ is a homterval for $g$, also (ii) is impossible (see Lemma~\ref{LemmaBaireXhomterval}), leading to a contradiction. Thus, we have that $A=A_0$.

Now we will prove that $A$ is a Cantor set, as we are assuming that $\interior(A)=\emptyset$. For that, we only need to show that $A$ does not admit isolated points.
Hence, assume by contradiction that $p\in A$ is a isolated point of $A$.
Let $I=(a,b)$ be the connected component of $[0,1]\setminus A$ containing $p$.
A $p\in \omega_g(x)$ generically on $U$ and $g$ is non-singular, then $\#\{x\in U\,;\,\co_g^+(x)\cap(a,p)\ne\emptyset\}=+\infty$ generically on $U$ or $\#\{x\in U\,;\,\co_g^+(x)\cap(p,b)\ne\emptyset\}=+\infty$  generically on $U$.
Suppose that $\#\{x\in U\,;\,\co_g^+(x)\cap(a,p)\ne\emptyset\}=+\infty$ generically on $U$, the other case is similar.

In this case, let $G:(a,p)^*\to(a,p)$ be the first return map to $(a,p)$ by $g$, where $(a,p)^*=\{x\in(a,p)\,;\,\co_g^+(g(x))\cap(a,p)\ne\emptyset\}$.
Thus, $U_0:=U\cap (a,p)\sim\bigcap_{n\ge0}G^{-n}(U)$ $=$ $\{x\in U\cap(a,p)\,;\,\#(\co_g^+(x)\cap(a,p))=\infty\}$ $\subset\bigcap_{n\ge0}G^{-n}((a,p))$ is a  fat set.
As a wandering interval of $G$ is a wandering interval of $g$, it follows from Lemma~\ref{LemmaINTERVALDICHOTOMY} and \ref{LemmaBaireXhomterval} that $\omega_G(x)=[a,p]$ generically on $(a,b)$. Thus, $\omega_G(x)=[a,p]$ generically on $U\cap(a,p)$. This implies that $\omega_g(x)$ is a cycle of interval generically on $U\cap(a,p)$ and so, by ergodicity, $\omega_g(x)=\omega_f(x)=A$ is a cycle of interval generically on $U$, contradicting our assumption that $\interior(A)=\emptyset$.
\end{proof}

\begin{Theorem}\label{TheoremFinTopInter}
Let $\cc$ be a finite subset of $(0,1)$ and 
$f:[0,1]\setminus\cc\to[0,1]$ a local homeomorphism.
If $f$ does not have wandering intervals and $\per(f)$ is a meager set, then 
either $\BB(f)$ is an open and dense subset of $[0,1]$ or there exists a finite collection of attractors $A_1,\cdots,A_\ell$, with $\ell\le\#\cc$, such that 
\begin{enumerate}
\item $[0,1]\setminus\BB(f)\sim\bigcup_{j=1}^{\ell}\beta_f(A_j)$;
\item for every $1\le j\le n$, $\omega_f(x)=A_j$ generically on $\beta_f(A_j)$;
\item each $A_j$ is either a Cantor set or a cycle of intervals;
\item if $A_j$ is a Cantor set then  $A_j=\bigcup_{v\in V}\overline{\co_f^+(v)}$, for some $V\subset\{f(c_\pm)\,;\,c\in\cc\}$ $($\footnote{ Indeed, $V=\{f(c_{\pm})\,;\,c\in\cc_{\pm}\}$, where $\cc_-=\{c\in\cc\,;\,c\in\overline{\co_f^+(x)\cap(0,c)}\}$ generically on $\beta_f(A_j)$ and $\cc_+=\{c\in\cc\,;\,c\in\overline{\co_f^+(x)\cap(c,1)}\}$ generically on $\beta_f(A_j)$.}$)$.
\end{enumerate}
\end{Theorem}
\begin{proof}
If $\BB(f)$ is not a dense set then $[0,1]\setminus\BB(f)$ is a fat set. 
As $f$ does not have wandering intervals and $\per(f)\sim\emptyset$, we get that $H(f)\sim\BB(f)$, where $H(f)$ is the union of all homtervals of $f$.
Thus, it follows from Proposition~\ref{PropositionOMEGAdecomposition} that $[0,1]\setminus\BB(f)$ can be decomposed into at most $1\le\ell\le\#\cc$ Baire ergodic components, say $U_1,\cdots, U_{\ell}$. 

Now, applying  Proposition~\ref{PropositionAttractorDichotomy}, each $U_j$ has a topological attractor $A_j$ and $\omega_f(x)=A_j$ generically on $\beta_f(A_j)\sim U_j$ (proving item (2)).
As $[0,1]\setminus\BB(f)\sim\bigcup_{j=1}^{\ell}\beta_f(A_j)$, we get item (1).
Items (3) and (4) follows also from Proposition~\ref{PropositionAttractorDichotomy}. 
\end{proof}

\begin{Corollary}\label{CorollaryFinTopInter}Let $\cc$ be a finite subset of $(0,1)$ and 
$f:[0,1]\setminus\cc\to[0,1]$ a local homeomorphism.
If $f$ is differentiable and $\limsup_n|(f^n)'(x)|>1$ in a dense subset of $[0,1]$,
then 
there exists a finite collection of attractors $A_1,\cdots,A_\ell$, with $\ell\le\#\cc$, such that 
\begin{enumerate}
\item $[0,1]\sim\bigcup_{j=1}^{\ell}\beta_f(A_j)$;
\item for every $1\le j\le n$, $\omega_f(x)=A_j$ generically on $\beta_f(A_j)$;
\item each $A_j$ is either a Cantor set or a cycle of intervals;
\item if $A_j$ is a Cantor set then  $A_j=\bigcup_{v\in V}\overline{\co_f^+(v)}$, for some $V\subset\{f(c_\pm)\,;\,c\in\cc\}$.
\end{enumerate}
\end{Corollary}
\begin{proof}
If $\per(f)$ is a fat set, then $\fix(f^n)$ is a fat set for some $n\in\NN$.
As $f$ is continuous, we get that $\fix(f^n)$ contains an open interval $I$ and so, $|(f^n)'(x)|=1$ for every $x\in I$, contradicting our hypothesis.
Thus, $\per(f)$ is a meager set.
As $\BB(f)$ is an open set and $\lim_n|(f^n)'(x)|<1$ for every $x\in\BB(f)$, we get that $\BB(f)=\emptyset$.
Likewise, as $\lim_n|(f^n)'(x)|\le1$ for every $x$ in a wandering interval, $f$ does not admit wandering intervals.
Therefore, $f$ satisfies the hypothesis of Theorem~\ref{TheoremFinTopInter}. 
\end{proof}


Proposition~\ref{PropositionWanderingIntervals} below appears implicitly in the proof of Theorem~A in \cite{BPP19}.

\begin{Proposition}[Wandering intervals attractors, \cite{BPP19}]\label{PropositionWanderingIntervals}
Let $f:[0,1]\setminus\cc\to[0,1]$ be a $C^2$ non-flat local diffeomorphism, with $\cc\subset[0,1]$ finite. There is $1\le\ell\le2^{2\#\cc}$ and a finite number of compact invariant sets $A_1,\cdots,A_{\ell}\subset[0,1]$ such that $\omega_f(I)=A_j$, for some $1\le j\le \ell$, whenever $I$ is wandering interval for $f$. Moreover, for each $1\le j\le\ell$,
\begin{enumerate}
\item $A_j\cap\cc\ne\emptyset$;
\item $A_j=\bigcup_{v\in V}\overline{\co_f^+(v)}$, for some $V\subset\{f(c_\pm)\,c\in\cc\}$.
\end{enumerate}
\end{Proposition}
\begin{proof}
As, if necessary, one can always extend $f$ to a $C^2$ map $g$ defined on bigger interval $J:=[a,b]\subset[0,1]$ so that $g(\partial J)\subset\partial J$ and $\omega_g(x)\in[0,1]$ for every $x\in(a,b)$ (see Figure~\ref{01PARAab.png}), we may assume $f(\{0,1\})\subset\{0,1\}$ and that  $\exists\,0<\alpha<\beta<1$ such that $\Omega(f)\cap(0,1)\subset[\alpha,\beta]$.
\begin{figure}
\begin{center}\includegraphics[scale=.19]{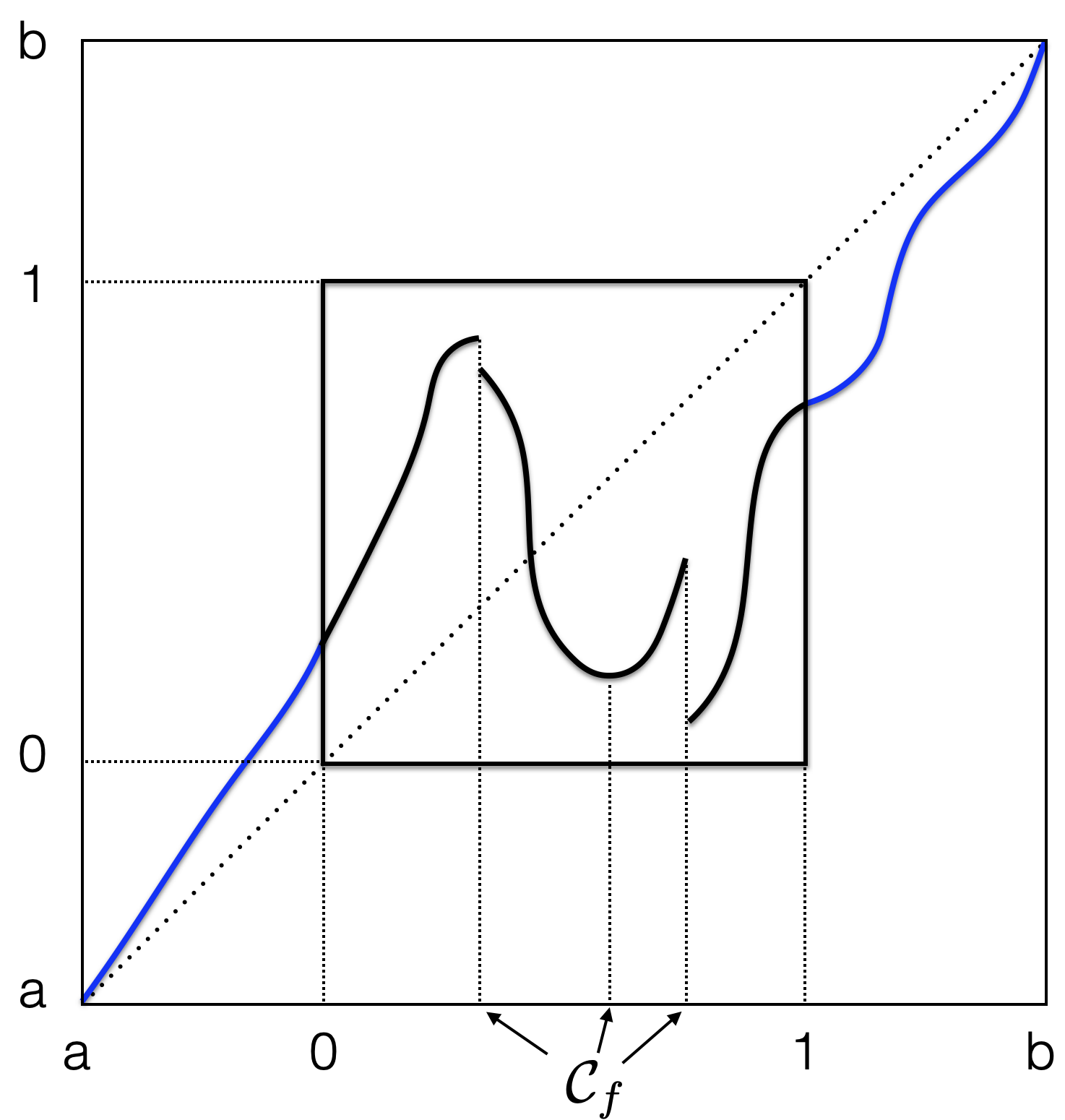}
\caption{One can see an extension of an original map $f$ defined on $[0,1]$ to a bigger interval $[a,b]$ in order to obtain an invariant boundary.} \label{01PARAab.png}
\end{center}
\end{figure}

Let $I$ be a wandering interval. As $\omega_f(I)=\omega_f(f^n(I))$ for every $n\ge0$, changing $I$ for $f^{\ell}(I)$ if necessary, we may assume that $\overline{f^n(I)}\cap\cc=\emptyset$ for every $n\ge0$.
Let $\cc_-(I)=\{c\in\cc\,;\,c\in\overline{(0,c)\cap\co_g^+(I)}\}$ and $\cc_+(I)=\{c\in\cc\,;\,c\in\overline{(c,1)\cap\co_g^+(I)}\}$.
For each $c\in\cc\setminus\cc_-(I)$, let $0<a_c<c$ be such that $\co_f^+(I)\cap(a_c,c)=\emptyset$.
Similarly, for each $c\in\cc\setminus\cc_+(I)$, let $c<b_c<1$ be such that $\co_f^+(I)\cap(c,b_c)=\emptyset$.
Let $$U=\bigg(\bigcup_{c\in\cc\setminus\cc_-(I)}(a_c,c)\bigg)\cup\bigg(\bigcup_{c\in\cc\setminus\cc_+(I)}(c,b_c)\bigg).$$ 
Let $g:[0,1]\setminus\cc\to[0,1]$ be a $C^2$ non-flat local diffeomorphism such that $g|_{([0,1]\setminus\cc)\setminus U}=f|_{([0,1]\setminus\cc)\setminus U}$ and $g(c_\pm)\in\{0,1\}$ for every $c\in\cc\setminus\cc_{\pm}(I)$ (see Figure~\ref{fPARAg.png}). 
Note that $f^n(I)=g^n(I)$ for every $n\ge0$ and so, $I$ is a wandering interval for $g$ and $\omega_f(I)=\omega_g(I)$. 

\begin{figure}
\begin{center}\includegraphics[scale=.3]{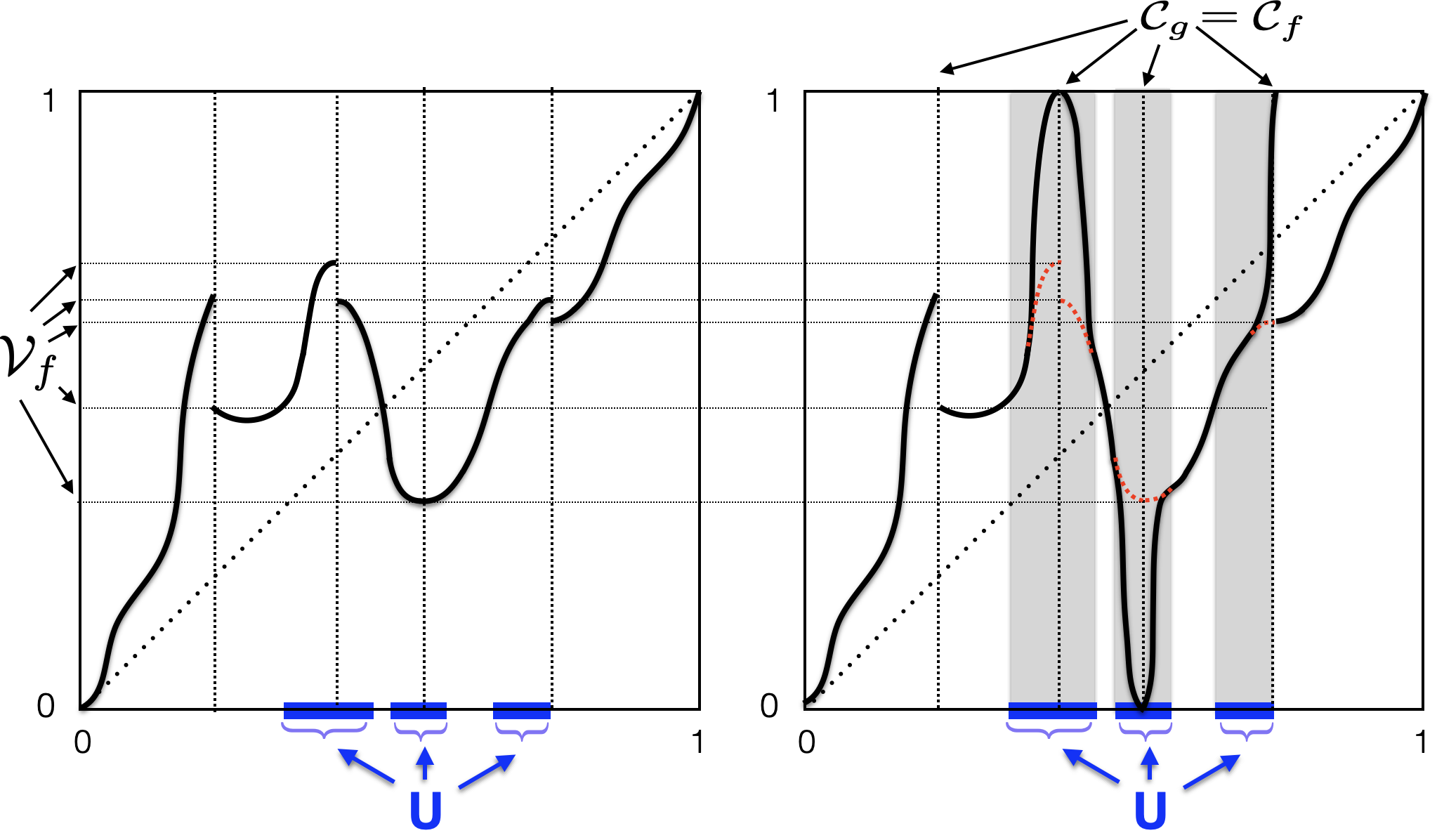}
\caption{On the right one can see the original map $f$ of the proof of Proposition~\ref{PropositionWanderingIntervals} and the modified map $g$ can be seen on the left side}\label{fPARAg.png}
\end{center}
\end{figure}

As $g(\{0,1\})=f(\{0,1\})\subset\{0,1\}$, we have that $\co_g^+(g(c_\pm))\subset\{0,1\}$ for every $c\in\cc\setminus\cc_\pm(I)$.
Hence, as $\omega_g(I)\subset\Omega(g)\cap(0,1)\subset[\alpha,\beta]\subset(0,1)$,
it follows from Proposition~6 of \cite{BPP19} (\footnote{\begin{prop}[Proposition~6 of \cite{BPP19}]
Given a $C^2$ non-flat local diffeomorphism $f:[0,1]\setminus\cc\to[0,1]$, $\cc\subset[0,1]$ is a finite set, let $\cv_f=\{f(c_{\pm})\,;\,c\in\cc\}$ be the set for critical values of $f$.
If $I$ is a wandering interval then $\omega_f(I)\subset\overline{\co_f^+(\cv_f)}$.
\end{prop}}) applied to $I$ and $g$ that $$\omega_g(I)\subset\bigg(\bigcup_{c\in\cc_-(I)}\overline{\co_g^+(g(c_-))}\bigg)\cup\bigg(\bigcup_{c\in\cc_+(I)}\overline{\co_g^+(g(c_+))}\bigg).$$
On the other hand, if $c\in\cc_-(I)$ then $g(c_-)\in\omega_g(I)$ and so, $\overline{\co_g^+(g(c_-))}\subset\omega_g(I)$. Similarly, $c\in\cc_+(I)$ $\implies$ $\overline{\co_g^+(g(c_+))}\subset\omega_g(I)$.
Therefore,  $$\omega_g(I)=\bigg(\bigcup_{c\in\cc_-(I)}\overline{\co_g^+(g(c_-))}\bigg)\cup\bigg(\bigcup_{c\in\cc_+(I)}\overline{\co_g^+(g(c_+))}\bigg).$$
As $\co_f^+(I)\cap U=\emptyset$, as well as $\co_f^+(f(c_\pm))\cap U=\emptyset$ $\forall c\in\cc_{\pm}(I)$, we get that 
$$\omega_f(I)=\bigcup_{v\in V}\overline{\co_f^+(v)},$$
where $V\subset\cv_f:=\{f(c_{\pm})\,;\,c\in\cc\}$ is given by $V=\{f(c_-)\,;\,c\in\cc_-(I)\}\cup\{f(c_+)\,;\,c\in\cc_+(I)\}$.
Thus, take $\mathfrak{A}=\{\bigcup_{s\in S}\overline{\co_f^+(s)}\,;\,S\in 2^{\cv_f}\}$ and let $A_1,\cdots,A_{\ell}$ be the set of $A\in\mathfrak{A}$ such that $\omega_f(I)=A$ for some wandering interval $I$.
As $\#\cv_f\le2\#\cc$, we conclude the proof of the proposition.
\end{proof}

\begin{proof}[\bf Proof of Theorem~\ref{TheoremFinTopInterC2}]
If $\BB(f)$ is not an open and dense set, the $[0,1]\setminus\BB(f)$ is a fat set. Let $\BB_1(f)$ be the union of all wandering intervals of $f$. If $\BB_1(f)\ne\emptyset$ let $\mathfrak{A}=\{ A_1',\cdots,A_{\ell}'\}$ ($\ell\le2^{2\#\cc}$)  be the topological attractors given by Proposition~\ref{PropositionWanderingIntervals}, otherwise let $\mathfrak{A}=\emptyset$.
If $\Omega(f)$ is a meager set, let $\mathfrak{B}=\emptyset$.
If $\Omega(f)$ is fat, it follows from Proposition~\ref{PropositionOMEGAdecomposition} that $\Omega(f)$ can be decomposed into at most $\#\cc$ ergodic components, say $U_1,\cdots,U_m$.
By Proposition~\ref{Propositiontop-ergodicAttractors}, for each $1\le j\le m$ there exists a topological attractor $A_1'',\cdots A_m''$ such that $U_j\sim\beta_f(A_j'')$ and $\omega_f(x)=A_j''$ generically on $U_j$.
So, if $\Omega(f)$ is not a meager set, let $\mathfrak{B}=\{A_1'',\cdots,A_m''\}$.
Hence, taking $\{A_1,\cdots,A_n\}=\mathfrak{A}\cup\mathfrak{B}$, we conclude the proof of $[0,1]\setminus\BB(f)\sim\bigcup_{j=1}^{\ell}\beta_f(A_j)$ and of items (1) and (2).
The proof of items (3), (4) and (5) follows from Proposition~\ref{PropositionAttractorDichotomy}.
\end{proof}

We observe that there are many basic open questions about attractors for maps of the intervals with discontinuities. One of them is the transitivity of those attractors. All the attractors for $C^2$ maps are transitive, but even for infinitely many renormalizable maps with discontinuities this is not clear, see a discussion about it in \cite{BPP20}.
Thus, to avoid more delicate problems, we focus in Theorem~\ref{TheoremUdecompodition} below only on $C^2$ non-degenerated maps.

There are three types of attractors for smooth interval maps: an attracting period orbit, a minimal (and uniquely ergodic) Cantor set and a cycle of interval (see Corollary~2.2 of \cite{BL}).
A {\em cycle of intervals} is a union of compact disjointed intervals $I_1,\cdots,I_t$ such that $f$ is transitive in $\bigcup_{j=1}^tI_j$ and, as we can see in Proposition~\ref{PropCyleIsStTr} below, it is a strongly transitive set for continuous maps. 

\begin{Proposition}\label{PropCyleIsStTr}
If $f:[0,1]\to[0,1]$ is continuous and $J=I_1\cup\cdots\cup I_s$ is a cycle of intervals then $f|_{J}$ is strongly  transitive.
\end{Proposition}
\begin{proof}
As $f$ is continuous $f(J)\subset J$, $I_1,\cdots,I_s$ is a collection of disjoints compact intervals and $f|_J$ is transitive, it follows that the first return map to $I_1$, $F:I_1\to I_1$, is given by $f^s$ and it is a transitive map.
As $I_1$ is compact and $F$ continuous, $F(I_1)\subset I_1$ is compact and, as $F$ is transitive, $F(I_1)=I_1$.
Let $(a,b)\subset I_1$.
If  $U=\bigcup_{n\ge0}F^n((a,b))\ne I_1$, let $V$ be a connect component of $U$.
As $F(U)\subset U$, for any $n\ge0$, either $F^n(V)\subset V$ or $F^n(V)\cap V=\emptyset$.
By transitiveness, there is a smaller $\ell\ge1$ such that $F^{\ell}(V)\subset V$.
This implies that $U=V\cup F(V)\cup\cdots F^{\ell-1}(V)$.

Let $\overline{F}$ be the continuous extension of $F$ to $[\alpha,\beta]:=\overline{V}$, that is,
 $\overline{F}:[\alpha,\beta]\to[\alpha,\beta]$ is given by $\overline{F}(x)=f^{\ell}(x)$.
As $\overline{F}$ is surjective and at least one point of $\partial V$ does not belong to $U$ (otherwise $U=I_1$), we get that $\{\alpha,\beta\}\setminus F((\alpha,\beta))\supset \{\alpha,\beta\}\setminus F(V)\ne\emptyset$ and this implies $\overline{F}^2(\alpha)=\alpha$ or $\overline{F}^2(\beta)=\beta$.   
Let us assume that $\overline{F}^2(\alpha)=\alpha$, the other case is similar.
In this case, we have that $\alpha\notin\overline{F}^2(V)$. 

Thus, it follows from the transitivity of
  $\overline{F}^2$ that there is $\varepsilon>0$ such that $\overline{F}^2(x)>x$ for every $\alpha<x<\alpha+\varepsilon$.
That is, $\alpha$ is a (topologically) repeller fixed points for $\overline{F}^2$.
This, together with the continuity of $\overline{F}^2$ and $\alpha\notin\overline{F}^2(V)$, implies that exists $r>0$ such that $\omega_{\overline{F}^2}(x)\subset[\alpha+r,\beta]$ for every $x\in(\alpha,\beta]$, breaking the transitivity of $\overline{F}^2$.
To avoid this contradiction, we must have $U=I_1$, proving that $F$ is strongly transitive. As a consequence, $f|_J$ is also strongly transitive. 
\end{proof}

A direct consequence of Theorem~B of \cite{Pi21}, Proposition~\ref{PropCyleIsStTr} above and the fact  that the periodic points are dense in any cycle of intervals for a continuous map $f$ is a formula for calculating the (upper) Birkhoff average for continuous function along the orbit of generic points in a cycle of intervals. 

\begin{Corollary}\label{CorollaryBirkForCycle}
If $J$ is a cycle of intervals for a continuous map $f:[0,1]\to[0,1]$ and $\varphi\in C([0,1],\RR)$ then $$\limsup\frac{1}{n}\sum_{j=0}^{n-1}\varphi\circ f^j(x)=\max\bigg\{\int\varphi d\mu\,;\,\mu\in\cm^1(f|_J)\bigg\}$$ and
$$\omega_f(x)=\omega_f^*(x)=J$$generically on $\beta_f(J)\sim\co_f^{-}(J)\subset\beta_f(J)$.
\end{Corollary}

Now we can give a quite complete description of the topological and statistical/ergodic behavior for generic orbits for non-degenerated $C^2$ maps.

\begin{proof}[\bf Proof of Theorem~\ref{TheoremUdecompodition}]
As $f$ does not have wandering interval (see \cite{MS89,MvS}) and as $\co_f^-(\per(f))$ is a countable set (since $\#\fix(f^n)<+\infty$ $\forall\,n\in\NN$), it follows from Theorem~\ref{TheoremFinTopInter} that there is a finite collection of topological attractors $A_1,\cdots,A_\ell$, $\ell\le\#\cc$ such that we get that $$[0,1]\sim\BB(f)\cup\bigcup_{j=1}^{\ell}\beta_f(A_j).$$
Each $A_j$ is either a Cantor set or a cycle of interval and $\omega_f(x)=A_j$ generically on $\beta_f(A_j)$. 
By \cite{MS89,MvS}, there exists $s_0\in\NN$ such that every  periodic point with period bigger than $s_0$ is an expanding periodic point.
Hence, $f$ can have at most $s_1=\#\big(\bigcup_{j=1}^{s_0}\fix(f^j)\big)<+\infty$ periodic attractors.
Let $\co_1,\cdots,\co_s$, $s\le s_1$, be the attracting periodic orbits of $f$.
That is, $\BB(f)=\beta_f(\co_1)\cup\cdots\cup\beta_f(\co_s)$.
Note also that $\omega_f^*(x)=\omega_f(x)$ for every $x\in\beta_f(\co_j)$ and $1\le j\le s$.
Moreover, $\nu_j=\frac{1}{\#\co_j}\sum_{p\in\co_j}\delta_p$ is the unique $f$ invariant ergodic probability $\mu$ such that $\mu(\co_j)>0$ and $\lim\frac{1}{n}\sum_{j=0}^{n-1}\varphi\circ f^j(x)=\int\varphi d\nu_j$ for every $x\in\beta_f(\co_j)$ and $\varphi\in C([0,1],\RR)$.

As $\omega_f(x)$ and $\omega_f^*(x)$ are  $f$-invariant forward subsets of $A_j$ for every $x\in\beta_f(A_j)$, if $A_j$ is a minimal set then $\omega_f(x)=\omega_f^*(x)=A_j$, $\forall\,x\in\beta_f(A_j)$, and this is the case when $A_j$ is an attracting periodic orbit or when $A_j$ is Cantor set (Corollary~2.2 \cite{BL}).
Thus, in both cases, the statistical and the topological attractor are the same, as well as $\beta_f^*(A_j)=\beta_f(A_j)$.
Moreover, if $A_j$ is a Cantor set the $A_j$ is a solenoid (see \cite{BL}), in particular, 
 $f|_{A_j}$ is uniquely ergodic.
 Hence,  $\beta_f^*(A_j)=\beta_f(\mu_j)=\beta_f(A_j)$ (\footnote{ The basin of attraction of a probability $\mu$, $\beta_f(\mu)$, is the set of all points $x\in[0,1]$ such that $\lim_n\frac{1}{n}\sum_{j=0}^{n-1}\delta_{f^j(x)}\to\mu$ in the weak* topology.}), where $\mu_j$ is the unique invariant probability for $f|_{A_j}$.
This implies that, when $A_j$ is a Cantor set or an attracting periodic orbit, the points $x\in\beta_f(A_j)$ do not have historical behavior and $\lim\frac{1}{n}\sum_{j=0}^{n-1}\varphi\circ f^j(x)=\int\varphi d\mu_j=\lambda_j=\max\{\int\varphi d\mu\,;\,\mu\in\cm^1(f)\text{ and }\mu(A_j)=1\}$.

Now, we have only to consider the remaining case, i.e., suppose that  $A_j=I_1\cup\cdots\cup I_s$ is a cycle of intervals.
A well known fact for a continuous piecewise monotone interval map $f$ is that the periodic points are dense in any cycle of intervals $J=J_1\cup\cdots\cup J_s$ (\footnote{ In general, this fact is not true when $f$ has discontinuities, even if $f$ is piecewise monotone.}).
Indeed, this fact follows from two results: (1) the topological entropy of the cycle $J$, $h_{top}(f|_{J})$, being positive (see for instance Corollary~4.6.11 in \cite{ALM} or the original proof by Block and Coven \cite{BlC}) and (2) the first return map $F$ to $J_1$ being conjugated to a continuous piecewise monotone interval map with constant slope $e^{s\,h_{top}(f)}>1$ (see for instance Corollary~4.6.12 in \cite{ALM}, this result was first proved by Parry \cite{Par}).

Hence, by Theorem~B of \cite{Pi21} and Proposition~\ref{PropCyleIsStTr}, we can conclude that generically $\omega_f^*(x)\supset A_j\cap\per(f)$.
Since $A_j=\overline{A_j\cap\per(f)}$, $\omega_f^*(x)=A_j$ generically on $\beta_f(A_j)$, proving that, also when $A_j$ is a the cycle of intervals,  the topological and the statistical attractors coincide.
Set $m=\ell+s$ and $A_{\ell+1}=\co_1,\cdots,A_{\ell+s}=\co_s$. From the fact that $\beta_f(A_1)\cup\cdots\cup\beta_f(A_{m})$ is a decomposition of $[0,1]$ up to a meager set and the fact that $\beta_f^*(A_j)\supset\beta_f(A_j)$ always, we get that $\beta_f^*(A_j)\sim\beta_f(A_j)$ for every $1\le j\le m$.
Finally, as item (6) is a consequence of Theorem~B (or Corollary~\ref{CorollaryBirkForCycle}), we conclude the proof of the theorem.
\end{proof}


\end{document}